\documentclass[reqno, a4paper, 12pt]{amsart}

\makeatletter
\let\origsection=\section
\def\section{\@ifstar{\origsection*}{\mysection}}
\def\mysection{\@startsection{section}{1}\z@{.7\linespacing\@plus\linespacing}{
.5\linespacing}{\normalfont\scshape\centering\S}}
\makeatother

\usepackage{comment}
\usepackage{amsmath,amssymb,amsthm}
\usepackage{mathrsfs}
\usepackage{dsfont}
\usepackage{mathabx}\changenotsign

\usepackage{color}
\usepackage[utf8]{inputenc}

\usepackage{xcolor}
\usepackage[backref=page]{hyperref}
\hypersetup{%
    colorlinks,
    linkcolor={red!60!black},
    citecolor={green!60!black},
    urlcolor={blue!60!black}
}

\usepackage{tikz}
\usepackage{subcaption}
\usetikzlibrary{calc,fit,shapes}
\usepackage{graphicx}
\usepackage{bookmark}

\usepackage[abbrev,msc-links,backrefs]{amsrefs}
\usepackage{doi}

\renewcommand{\PrintDOI}[1]{\doi{#1}}

\usepackage[T1]{fontenc}
\usepackage{lmodern}
\usepackage[english]{babel}

\usepackage{stmaryrd}

\usepackage{fullpage}
\usepackage{setspace}
\usepackage{comment}
\usepackage{enumitem}

\let\polishlcross=\l
\def\l{\ifmmode\ell\else\polishlcross\fi}

\let\emptyset=\varnothing
\let\setminus=\smallsetminus

\newcommand{\oldqed}{}
\def\endofFact{\hfill\scalebox{.6}{$\Box$}}

\makeatletter
\def\moverlay{\mathpalette\mov@rlay}
\def\mov@rlay#1#2{\leavevmode\vtop{%
   \baselineskip\z@skip \lineskiplimit-\maxdimen
   \ialign{\hfil$\m@th#1##$\hfil\cr#2\crcr}}}
\newcommand{\charfusion}[3][\mathord]{
    #1{\ifx#1\mathop\vphantom{#2}\fi
        \mathpalette\mov@rlay{#2\cr#3}
      }
    \ifx#1\mathop\expandafter\displaylimits\fi}
\makeatother

\newtheorem{theorem}{Theorem}
\newtheorem{lemma}[theorem]{Lemma}

\newtheorem{problem}[theorem]{Problem}

\newtheorem{proposition}[theorem]{Proposition}

\usepackage{datetime}
\usepackage{lineno}

\makeatletter
\newcommand{\definetitlefootnote}[1]{%
  \newcommand\addtitlefootnote{%
    \makebox[0pt][l]{$^{*}$}%
    \footnote{\protect\@titlefootnotetext}
  }%
  \newcommand\@titlefootnotetext{\spaceskip=\z@skip $^{*}$#1}%
}
\makeatother

\begin{document}
%\linenumbers 

\onehalfspace
%\setstretch{1.18}
\shortdate
\yyyymmdddate
\settimeformat{ampmtime}
\date{\today, \currenttime}

\newcommand{\Nat}{{\mathbb N}}
\newcommand{\Real}{{\mathbb R}}

\newcommand{\mcarrow}{\overset{\mathrm{rb}}{{\longrightarrow}}}
\newcommand{\nmcarrow}{\overset{\mathrm{rb}}{{\longarrownot\longrightarrow}}}

\newtheorem{case}{Case}

\definetitlefootnote{%
  An extended abstract of this work appeared in the proceedings of LAGOS 2019 published in~\cite{LAGOSpaper}
 }
    
\title{Anti-Ramsey threshold of cycles}

    \author[G. F. Barros]{Gabriel Ferreira Barros}
    \author[B. P. Cavalar]{Bruno Pasqualotto Cavalar}
    \author[G. O. Mota]{Guilherme Oliveira Mota}
    \author[O. Parczyk]{Olaf Parczyk}
    
    \address{Institute of Mathematics and Statistics, University of S{\~a}o Paulo, S{\~a}o Paulo, Brazil}
    \email{\{gbarros | brunopc | mota\} @ime.usp.br}

    \address{London School of Economics, Department of Mathematics, London, WC2A 2AE, UK.}
\email{o.parczyk@lse.ac.uk}

\thanks{
G. F. Barros was partially supported by CAPES.
    B. P. Cavalar was partially supported by FAPESP (Proc.~2018/05557-7).
    G. O. Mota was partially supported by CNPq (304733/2017-2, 428385/2018-4) and FAPESP (2018/04876-1, 2019/13364-7).
    O.~Parczyk was partially supported by Technische Universit\"at Ilmenau, the Carl Zeiss Foundation, and the DFG (Grant PA 3513/1-1).
    The collaboration of the authors was supported by CAPES/DAAD PROBRAL (Proc. 430/15).
    This study was financed in part by the Coordenação de Aperfeiçoamento de Pessoal de Nível Superior, Brasil (CAPES), Finance Code 001.
}

\begin{abstract} 
  For graphs~$G$ and~$H$, let~$G \mcarrow H$
  denote the property that
  for every \emph{proper} edge colouring of~$G$ there is a \emph{rainbow} copy of~$H$ in~$G$.
  Extending a result of Nenadov, Person, \v{S}kori\'{c} and Steger~(2017),
  we determine the threshold for $G(n,p) \mcarrow C_\ell$ for cycles $C_\ell$ of any given length $\ell\geq 4$.
\end{abstract}

\maketitle

\section{Introduction}\label{intro}

In this paper we investigate an \emph{anti-Ramsey} property of random graphs.
Given graphs~$G$ and~$H$, we denote by~$G \mcarrow H$ the following anti-Ramsey property:
for every \emph{proper} edge colouring of~$G$ there is a \emph{rainbow} copy of~$H$ in~$G$,
i.e.\ a subgraph of $G$ isomorphic to~$H$ in which all edges have distinct colours.

In 1992, R\"odl and Tuza~\cite{RoTu92} proved the following result, which answered affirmatively a question raised by Spencer (see~\cite[p.~29]{Er79}) asking whether there are graphs of arbitrarily
large girth containing a rainbow cycle in every proper edge colouring.

\begin{theorem}[\cite{RoTu92}]
	For every positive integer $t$ and every positive $\delta$ with $\delta<1/(2t+1)$ there exists $n_0$ such that for every $n\geq n_0$ there exists an $n$-vertex graph $G$ with girth at least $t+2$ having the property $G\mcarrow C_\ell$, for $2t+1\leq \ell\leq n^\delta$, where $C_\ell$ is an $\ell$-vertex cycle.
\end{theorem}

In their proof, R\"odl and Tuza showed that $G(n,p) \mcarrow C_\ell$ holds a.a.s.\footnote{A property in $G(n,p)$ holds asymptotically almost surely (a.a.s.) if the probability tends to one as $n$ tends to infinity.}~for a small~$p$.
Note that since~$G \mcarrow H$ is an increasing property,
there exists a threshold\footnote{The threshold for a property is a function $\hat{p}=\hat{p}(n)$ such that $G(n,p)$ a.a.s.~has this property if $p \gg \hat{p}$ and a.a.s.~does not have it if $p \ll \hat{p}$.} $p_H^{\text{rb}}=p_H^{\text{rb}}(n)$ for any fixed graph $H$ (see~\cite{BoTh87}).
In~\cite{KoKoMo12a}, Kohayakawa, Konstadinidis and Mota obtained an upper bound for the threshold $p_{H}^{\text{rb}}$ for any fixed graph~$H$ in terms of the maximum $2$-density $m_2(H) = \max\left\{(e(J)-1)/(v(J)-2)
: J\subseteq H, v(J)\ge 3 \right\}$.

\begin{theorem}[\cite{KoKoMo12a}] \label{upper}
	Let~$H$ be a fixed graph.
	Then there exists a constant~$C>0$
	such that~a.a.s.\ $G(n,p) \mcarrow H$
	whenever~$p=p(n)\ge C n^{-1/m_2(H)}$.
	In particular,~$p_H^{\text{rb}} \le n^{-1/m_2(H)}$.
\end{theorem}

A classical result in Ramsey Theory obtained by R\"{o}dl and Ruci\'{n}ski~\cite{RoRu95} implies that $n^{-1/m_2(H)}$ is the threshold for the following Ramsey property, as long as $H$ contains a cycle: every colouring of $E(G(n,p))$ with $r$ colours contains a monochromatic copy of $H$.
In view of this result, it is plausible to conjecture that $n^{-1/m_2(H)}$ is also the threshold for the anti-Ramsey property, for any fixed graph $H$.
However, as proved in~\cite{KoKoMo12b}, there are infinitely many graphs~$H$ for which
the threshold~$p_H^{\text{rb}}$ is asymptotically smaller than~$n^{-1/m_2(H)}$.
Recently, this result was extended to a larger family of graphs~(see~\cite{AKM5}).
On the other hand, Nenadov, Person, \v{S}kori\'{c} and Steger~\cite{NePeSkSt17} proved that
at least for sufficiently large cycles and complete graphs $H$
the lower bound for $p_H^{\text{rb}}$ matches
the upper bound $n^{-1/m_2(H)}$ of Theorem~\ref{upper}.

\begin{theorem}[\cite{NePeSkSt17}] \label{nenadov}
	If~$H$ is a cycle on at least~$7$ vertices or a complete graph on at least~$19$ vertices,
	then~$p_H^{\text{rb}}=n^{-1/m_2(H)}$.
\end{theorem}

In~\cite{KoMoPaSc18+}, Kohayakawa, Mota, Parczyk and Schnitzer extended Theorem~\ref{nenadov}, by showing that for all complete graphs~$K_\ell$ with~$\ell\geq 5$ the threshold $p_{K_\ell}^{\text{rb}}$ is in fact~$n^{-1/m_2(K_\ell)}$,
and for $K_4$ we have~$p_{K_4}^{\text{rb}}=n^{-7/15} \ll n^{-1/m_2(K_4)}$.
Our result determines the threshold~$p_{C_\ell}^{\text{rb}}$
for every cycle $C_\ell$ on~$\ell \ge 4$ vertices.

\begin{theorem} \label{main}
	Let~$\ell \ge 5$ be an integer. Then $p_{C_\ell}^{\text{rb}} = n^{-1/m_2(H)}$.
	Furthermore, $p_{C_4}^{\text{rb}} = n^{-3/4}$.
\end{theorem}

In Section~\ref{sec:cycles} we prove Theorem~\ref{main} for cycles with at least~$5$ vertices.
Similarly to what happens with complete graphs, the situation for~$C_4$ is different: For $p=n^{-3/4}\ll n^{-1/m_2(C_4)}$ the random graph $G(n,p)$ a.a.s.\ contains a small graph $F$ such that $F\mcarrow C_4$.
In Section~\ref{sec:C4} we prove that~$p_{C_4}^{\text{rb}} = n^{-3/4}$ gives the treshold for $C_4$\footnote{We remark that a sketch of the proof for $C_4$ was given in a short abstract of the fourth author~\cite{Mo17}.} and we finish with some concluding remarks in Section~\ref{sec:con}.
We use standard notation and terminology (see e.g.~\cite{Di10} and~\cite{JLR00}).
In particular, given a subgraph $H$ of a graph $G$, we write~$G-H$ for the graph obtained from~$G$ by removing all vertices that belong to~$H$ and all edges incident with these vertices.

\section{Cycles on at least five vertices} \label{sec:cycles}

In~\cite{NePeSkSt17}, Nenadov, Person, \v{S}kori\'{c}, and Steger provide a general framework that reduces some Ramsey problems into deterministic problems for graphs with bounded maximum density, where the 
\emph{maximum density} of a graph~$G$ is denoted 
% by~$m(G)=\max\left\{e(J)/v(J) : J\subseteq G, v(J)\ge 1 \right\}$.
by
\begin{equation*}
m(G)=\max\left\{ \frac{e(J)}{v(J)} : J\subseteq G, v(J)\ge 1 \right\}.
\end{equation*}
The proof of~Theorem~\ref{nenadov} for cycles relies on the following lemma (see~\cite[Lemma~24]{NePeSkSt17}).

\begin{lemma}[\cite{NePeSkSt17}] \label{nenadov_lemma}
	Let~$\ell \ge 7$ be an integer and~$G$ be a graph such that~$m(G) < m_2(C_\ell)$.
	Then~$G \nmcarrow C_\ell$.
\end{lemma}

In fact they prove a slightly stronger statement for which they need a non-strict inequality relating the densities~\cite[Corollary~13]{NePeSkSt17}.
The condition~$\ell \ge 7$ in Theorem~\ref{nenadov} is simply a consequence
of the restriction on the cycle length imposed in Lemma~\ref{nenadov_lemma}, as observed by the authors~\cite{NePeSkSt17}.
We extend Lemma~\ref{nenadov_lemma}, proving the following result, where we note that~$m_2(C_\ell) = (\ell-1)/(\ell-2)$.

\begin{lemma}\label{lemma_C5}
	Let~$\ell \ge 5$ be an integer and~$G$ be a graph such that~$m(G) < (\ell-1)/(\ell-2)$.
	%Then there exists a proper edge colouring of $G$ such that every~$\ell$-cycle
	%contains two non-adjacent edges with the same colour.
	Then,~$G \nmcarrow C_\ell$
\end{lemma}

Theorem~\ref{main} thus follows immediately by replacing Lemma~\ref{nenadov_lemma}
with our Lemma~\ref{lemma_C5} in the proof of Theorem~\ref{nenadov} in~\cite{NePeSkSt17}.
We remark that the proof of Lemma~\ref{lemma_C5} considers all the cycle lengths in the range $\ell \ge 5$,
i.e.\ it is not a proof only for the cases~$\ell=5$ and~$\ell=6$.

Throughout this section let $\ell \ge 5$ be an integer and~$G$ be a graph with~$m(G)<(\ell-1)/(\ell-2)$.
We use the term \emph{$k$-path} to refer to a path with $k$ vertices.
For the proof of Lemma~\ref{lemma_C5}, we will define
a \emph{partial} proper edge colouring of $G$ such that
every~$\ell$-cycle has two non-adjacent edges with the same colour.
Clearly, having defined such a partial edge colouring, we can extend it
to a proper edge colouring (for instance, the uncoloured edges may
be assigned distinct colours).

\subsection{Cycle components}

Let~$\mathcal{C}_{\ell}(G)$ be the set of all~$\ell$-cycles of~$G$.
We start by defining key concepts that we use throughout our proof.
The \emph{edge intersection graph} of~$\mathcal{C}_{\ell}(G)$ is the graph whose
vertex set is~$\mathcal{C}_{\ell}(G)$ and whose edges correspond to
pairs~$\{C,C'\}$ such that~$C \neq C'$ and~$E(C) \cap E(C') \neq \emptyset$.
A subgraph~$H \subseteq G$ is a~\emph{$C_{\ell}$-component} of~$G$
if it is the union of all~$\ell$-cycles corresponding to the vertices of some component of the
edge intersection graph of~$\mathcal{C}_{\ell}(G)$.

Let~$H_1$ be an~$\ell$-cycle in~$G$.
A~$C_{\ell}$-component $H$ of~$G$ containing~$H_1$ can be constructed from~$H_1$ as follows.
Suppose we have defined $H_1 \subseteq \dots \subseteq H_i$ for $i \ge 1$.
If there is an~$\ell$-cycle~$C$ in~$G$ such that~$C \not \subseteq H_i$
and~$E(C) \cap E(H_i) \neq \emptyset$,
then we put $H_{i+1} = H_i \cup C$; otherwise
we terminate the construction and set $H = H_i$.
Let $t$ be such that $H=H_t$.
We call $(H_1,\ldots,H_t)$ a \emph{construction sequence} of~$H$.
For brevity, sometimes we will identify a~$C_\ell$-component
with a construction sequence of it; for example, we will write
``a~$C_\ell$-component~$(H_1,\dots,H_t)$''.

Note that there can be multiple new $\ell$-cycles appearing in $H_{i+1}$ that were not present in $H_i$ before; this will be the main problem to deal with when constructing the partial colouring.
Also note that the process just described allows us to reconstruct
a~$C_\ell$-component starting from any~$\ell$-cycle of it.
Also note that two~$\ell$-cycles belonging to distinct~$C_{\ell}$-components
may share vertices (obviously they do not share edges).

We start the colouring procedure in some~$C_{\ell}$-component~$H$ of~$G$.
Once we have coloured the edges of~$H$ avoiding a rainbow $C_\ell$, we proceed to assign colours
different from those used in~$H$ to edges of a~$C_{\ell}$-component of $G-E(H)$, using the same procedure.
We continue colouring edges in this manner (taking an uncoloured~$C_\ell$-component, colouring it and removing its edges) until we have considered all the $\ell$-cycles of~$G$.
Thus, our aim is to describe the colouring procedure of an
arbitrary~$C_{\ell}$-component~$H$ of~$G$.

Let~$(H_1,\ldots,H_t)$ be a $C_\ell$-component of~$G$.
Since producing a colouring which avoids rainbow~$C_\ell$
is a trivial task if the $C_\ell$-component
has only one cycle, we may assume~$t \ge 2$.
The following proposition is crucial in our proof
and, given a $C_\ell$-component $(H_1,\ldots, H_t)$, describes for any 
$1 \le i \le t-1$ 
the possible structure of an $\ell$-cycle~$C$
which is added to~$H_i$ to form~$H_{i+1}$,
i.e.\ $C \subseteq H_{i+1}$, but~$C \not \subseteq H_i$
and~$E(C) \cap E(H_i) \neq \emptyset$.
(see Figure~\ref{fig:config}).

\begin{proposition} \label{prop_config}
	Let $\ell\geq 5$ be an integer, $G$ be a graph with $m(G)<(\ell-1)/(\ell-2)$ and $(H_1,\ldots, H_t)$ be a $C_\ell$-component of $G$.
	Then, the following holds for every $1 \le i \le t-1$.
	
	If~$C$ is an~$\ell$-cycle added to~$H_{i}$ to form~$H_{i+1}$,
	then there exists a labelling $C=u_1u_2\cdots u_\ell u_1$ such that
	exactly one of the following occurs, where $2 \le k \le \ell$ and $3 \le j \le \ell-1$:
	
	\begin{enumerate}[itemindent=1em]
		\item[$(A_k)$]\label{config_A}
		$u_1u_2\cdots u_k$ is a~$k$-path in~$H_{i}$ and $u_{k+1},\dots,u_\ell \notin V(H_i)$; 
		
		\item[$(B_j)$] \label{config_B}
		$u_1u_2 \in E(H_i)$, $u_2u_3 \notin E(H_i)$, $\{u_3,\dots, u_\ell\} \setminus \{u_j\} \subseteq V(H_{i+1})\setminus V(H_i)$, $u_j \in V(H_i)$.
	\end{enumerate}
\end{proposition}

\tikzset{every picture/.style={line width=0.75pt}} %set default line width to 0.75pt        
\begin{figure}
	\centering
	\input{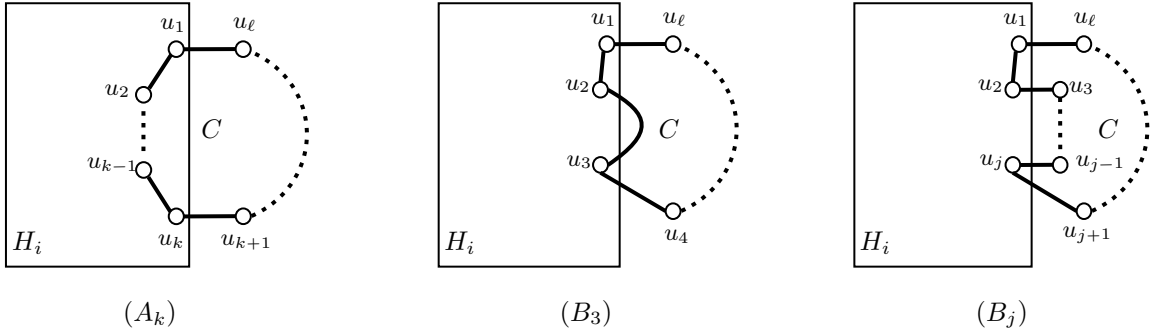}
	\caption{Possible configurations of a $C_\ell$ added to $H_i$ to form $H_{i+1}$.}
	\label{fig:config}
\end{figure}

We refer to each of $(A_k)$ and $(B_j)$ as a \emph{configuration} of $H_{i+1}$.
Before proving Proposition~\ref{prop_config}, let us discuss some ideas used for this purpose.
To show that some of the configurations are not possible or do not happen often during the construction of $(H_1,\ldots, H_t)$, we heavily use the fact that~$m(G)<(\ell-1)/(\ell - 2)$.

For any $1 \le j \le i$, define parameters $e_j$, $v_j$ and $c_j$ as follows:
$e_j$ is the number of edges in~$E(H_{j+1})\setminus E(H_j)$, while~$v_j$
stands for the number of vertices in~$V(H_{j+1})\setminus V(H_j)$.
Lastly, let~$c_j$ be the number of components of $H_{j+1}-H_j$.
Note that if~$v_j = 0$, then~$e_j \ge 1$,
and if~$v_j \ge 1$, then the components of~$H_{j+1}-H_j$ are paths and we get~$e_j \ge v_j+c_j \ge v_j+1$.
Therefore, we conclude that, for $1 \le j \le i$ we have 
$
e_j\geq v_j + 1
$
Also, since any $\ell$-cycle added to $H_j$ to form $H_{j+1}$ contains at least one edge of $H_j$, for $1\leq j\leq i$, we have
$
v_j \le \ell-2.
$
Note that we have
\begin{equation}
\label{eq_density}
\frac{\ell-1}{\ell-2} > m(G) \geq  \frac{e(H_{i+1})}{v(H_{i+1})} = 
\frac{\ell+\sum_{j=1}^{i}e_j}{\ell+\sum_{j=1}^{i}v_j}.
\end{equation}
%By an adequate number of applications of the inequality $a/b > (a+1)/(b+1)$ and 
Using the bounds $e_j\geq v_j + 1$ and $v_j \le \ell-2$, we obtain 
\begin{equation}
\label{eq_density2}
\frac{\ell-1}{\ell-2}
>  \frac{\ell+e_i+\sum_{j=1}^{i-1}(v_j+1)}{\ell+v_i+\sum_{j=1}^{i-1}v_j}
\ge \frac{\ell+e_i+(i-1)(\ell-1)}{\ell+v_i+(i-1)(\ell-2)},
\end{equation}
which implies
\begin{equation}
\label{eq_density-easy}
e_i < \frac{(\ell-1)v_i + \ell}{\ell-2}.
\end{equation}
We are ready to prove Proposition~\ref{prop_config}.

\begin{proof}[Proof of Proposition~\ref{prop_config}]
	% Let $\ell\geq 5$, $G$, $(H_1,\ldots,H_t)$, $i$ and $C=u_1u_2\cdots u_\ell u_1$ as in the statement.
	We will prove the result for all possible values of $v_i$ (i.e., $0\leq v_i\leq \ell-2$).
	If $v_i=\ell-2$, then we have configuration~\hyperref[config_A]{$(A_{2})$}.
	
	Now let $v_i=\ell-3$, which means that there are exactly three vertices of $C$ in $H_i$.
	If these vertices form a path, then we have configuration~\hyperref[config_A]{$(A_3)$}.
	On the other hand, let $u_1$, $u_2$ and $w$ be the vertices of $C$ in $H_i$ and let $u_1u_2$ be an edge of $H_i$.
	If there is an edge of $C$ between $w$ and $\{u_1,u_2\}$, then let w.l.o.g.\ $u_2w$ be this edge.
	Then, we have configuration~\hyperref[config_B]{$(B_{3})$}, where $u_3=w$.
	It there is no edge of $C$ between $w$ and $\{u_1,u_2\}$, then w.l.o.g.\
	$C$ contains a path $P_1=u_2,u_3,\ldots,u_{j-1},w$ 
	(with at least two edges) between $u_2$ and $w$
	with all edges outside $H_i$, and a path $P_2=w,u_{j+1},\ldots,u_\ell,u_1$
	between $w$ and $u_1$ with all edges outside $H_i$,
	such that $w$ is the only common vertex of $P_1$ and $P_2$.
	Then, we have configuration~\hyperref[config_B]{$(B_{j})$}, where $u_j=w$ and $4\leq j\leq \ell-1$ (as $u_3$ and $u_\ell$ are vertices outside $H_i$).
	
	Finally, let~$0 \le v_i \le \ell-4$.
	From~\eqref{eq_density-easy} we have~$e_i \le v_i+1$. 
	Then, $H_{i+1} - H_i$ has only one component, which implies that the vertices of $C$ in $H_i$ form a path of length $\ell-v_i$, where we have $4\leq \ell-v_i\leq \ell$.
	Therefore, we have configuration~\hyperref[config_A]{$(A_{k})$} with $4 \le k \le \ell$.
\end{proof}

\subsection{Proof of Lemma~\ref{lemma_C5}}

Given a $C_\ell$-component $H$ described by a construction sequence $(H_1,\ldots, H_t)$, we will colour the edges of $H_1$, $H_2$ and so on iteratively, avoiding rainbow $\ell$-cycles.
For configurations~\hyperref[config_A]{$(A_{k})$} with $2 \le k \le \ell-2$ we are always able to assign a new colour $i$ to two non-adjacent new edges.
All other configurations may appear at most twice in $(H_1,\ldots, H_t)$, and in these cases we will colour all previous configurations carefully so that we are able to proceed.

Arguments involving calculations similar to those we did on~\eqref{eq_density} and~\eqref{eq_density2} will be referred to as \emph{density arguments}.
For example, when $H_i$ has configuration~\hyperref[config_A]{$(A_\ell)$}, we have $v_i=0$ and $e_i=1$, which following the calculations in \eqref{eq_density} and~\eqref{eq_density2} implies that there cannot be another occurrence of~\hyperref[config_A]{$(A_\ell)$}, as this would imply
\begin{equation*}
\frac{\ell-1}{\ell-2} > m(G) \geq  \frac{e(H_{i+1})}{v(H_{i+1})} = 
\frac{\ell+\sum_{j=1}^{i}e_j}{\ell+\sum_{j=1}^{i}v_j} \geq
\frac{\ell+2+(i-3)(\ell-1)}{\ell + (i-3)(\ell-2)},
\end{equation*}
which gives the following contradiction, as $\ell\geq 5$:
\begin{equation*}
\ell(\ell-1) > (\ell-2)(\ell+2).
\end{equation*}
Similarly, one can show that configuration~\hyperref[config_A]{$(A_{\ell-1})$}, where $v_i=1$ and $e_i=2$, appears at most twice and any~\hyperref[config_B]{$(B_{j})$}, where $v_i=\ell-3$ and $e_i=\ell-1$, at most once.
Furthermore, when one of these configurations appears, the occurrence of~\hyperref[config_A]{$(A_k)$} with $3 \le k \le \ell-2$ is restricted, while only~\hyperref[config_A]{$(A_2)$} can appear arbitrarily often.
We will refer to these estimates as the \emph{density argument}.

\begin{proof}[Proof of Lemma~\ref{lemma_C5}]
	Let~$\ell \ge 5$ be an integer and~$G$ be a graph such that~$m(G) < (\ell-1)/(\ell-2)$.
	Choose an arbitrary~$\ell$-cycle~$H_1$ in $G$ and assign a colour $c_1$ to a pair
	of non-adjacent edges of~$H_1$. Let~$H=H_t$, with~$t \ge 2$, be the~$C_\ell$-component
	of~$G$ obtained from a construction sequence $(H_1,\ldots, H_t)$.
	
	Now we consider a few cases according to which configurations given by Proposition~\ref{prop_config}
	occur in $(H_1,\ldots, H_t)$. For each~$1 \le i \le t-1$, note that there can be many cycles in $H_{i+1}$ that are not in $H_i$.
	We will assign colours to the edges of $E(H_{i+1})\setminus E(H_i)$ such that in $H_{i+1}$ any $\ell$-cycle has two edges coloured with the same colour.
	
	Since the connected components of $H_{i+1}-H_i$ are paths, in case each of these paths contains two vertices, 
	we can give a new colour $c$ to two non-adjacent edges of $E(H_{i+1})\setminus E(H_i)$.
	Then, any $\ell$-cycle of $H_{i+1}$ that contains these paths becomes non-rainbow (see~Figure~\ref{fig:cases}-(a)).
	If $H_{i+1}$ has configuration~\hyperref[config_A]{$(A_{k})$} with $2 \le k \le \ell-2$, this is how we proceed, unless stated otherwise.
	But it may be the case that $H_{i+1}$ contains an $\ell$-cycle that is not in $H_i$ and it does not contains such paths (it can be formed with edges between vertices of $H_i$ and components of $H_{i+1}-H_i$ of only one vertex (see~Figure~\ref{fig:cases}-(b)) and we have to be more careful colouring these edges.
	\tikzset{every picture/.style={line width=0.75pt}} %set default line width to 0.75pt        
	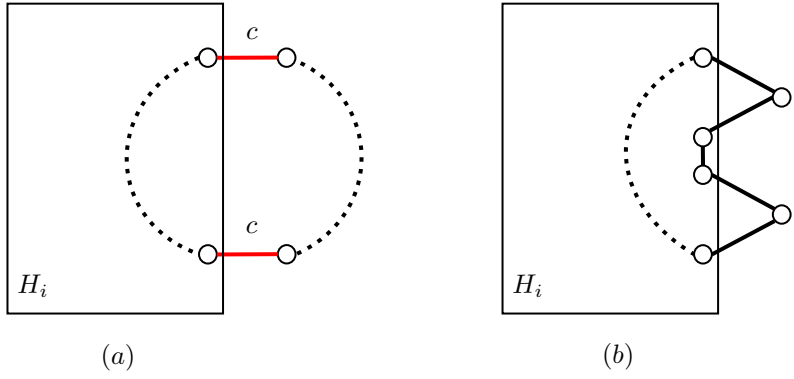
\begin{figure}
		\centering
		\tikzset{every picture/.style={line width=0.75pt}} %set default line width to 0.75pt        

\begin{tikzpicture}[x=0.75pt,y=0.75pt,yscale=-1,xscale=1]
%uncomment if require: \path (0,399); %set diagram left start at 0, and has height of 399

%Shape: Ellipse [id:dp8316233218480076] 
\draw  [line width=0.75]  (115.32,47.62) .. controls (115.35,45.07) and (117.35,43.02) .. (119.8,43.05) .. controls (122.25,43.08) and (124.21,45.18) .. (124.18,47.73) .. controls (124.15,50.28) and (122.14,52.33) .. (119.7,52.3) .. controls (117.25,52.27) and (115.29,50.17) .. (115.32,47.62) -- cycle ;
%Shape: Ellipse [id:dp551475780936526] 
\draw  [line width=0.75]  (154.62,47.62) .. controls (154.65,45.07) and (156.66,43.02) .. (159.1,43.05) .. controls (161.55,43.08) and (163.51,45.18) .. (163.48,47.73) .. controls (163.45,50.28) and (161.45,52.33) .. (159,52.3) .. controls (156.55,52.27) and (154.59,50.17) .. (154.62,47.62) -- cycle ;
%Shape: Ellipse [id:dp19554944412087139] 
\draw   (115.37,146.57) .. controls (115.4,144.02) and (117.41,141.97) .. (119.85,142) .. controls (122.3,142.03) and (124.26,144.12) .. (124.23,146.68) .. controls (124.2,149.23) and (122.2,151.27) .. (119.75,151.24) .. controls (117.3,151.21) and (115.34,149.12) .. (115.37,146.57) -- cycle ;
%Shape: Ellipse [id:dp7942702120473546] 
\draw   (154.67,146.52) .. controls (154.7,143.97) and (156.71,141.92) .. (159.16,141.95) .. controls (161.6,141.98) and (163.56,144.08) .. (163.53,146.63) .. controls (163.51,149.18) and (161.5,151.23) .. (159.05,151.2) .. controls (156.61,151.17) and (154.65,149.07) .. (154.67,146.52) -- cycle ;
%Straight Lines [id:da43965933406583324] 
\draw [color={rgb, 255:red, 255; green, 0; blue, 0 }  ,draw opacity=1 ][fill={rgb, 255:red, 255; green, 0; blue, 0 }  ,fill opacity=1 ][line width=1.5]    (124.18,47.73) -- (154.62,47.62) ;
%Straight Lines [id:da9392389376302626] 
\draw [color={rgb, 255:red, 255; green, 0; blue, 0 }  ,draw opacity=1 ][fill={rgb, 255:red, 255; green, 0; blue, 0 }  ,fill opacity=1 ][line width=1.5]    (124.23,146.68) -- (154.67,146.52) ;
%Shape: Rectangle [id:dp22695402872558335] 
\draw   (19.8,20.5) -- (127.35,20.5) -- (127.35,176.39) -- (19.8,176.39) -- cycle ;
%Curve Lines [id:da22743886661503243] 
\draw [line width=1.5]  [dash pattern={on 1.69pt off 2.76pt}]  (166.09,50.39) .. controls (207.01,70.63) and (207.01,127.3) .. (163.53,146.63) ;
%Curve Lines [id:da5572553096011394] 
\draw [line width=1.5]  [dash pattern={on 1.69pt off 2.76pt}]  (115.32,47.62) .. controls (67,68.89) and (67.5,129.39) .. (115.37,146.57) ;
%Shape: Ellipse [id:dp17706993403460403] 
\draw  [line width=0.75]  (362.32,47.62) .. controls (362.35,45.07) and (364.35,43.02) .. (366.8,43.05) .. controls (369.25,43.08) and (371.21,45.18) .. (371.18,47.73) .. controls (371.15,50.28) and (369.14,52.33) .. (366.7,52.3) .. controls (364.25,52.27) and (362.29,50.17) .. (362.32,47.62) -- cycle ;
%Shape: Ellipse [id:dp46347465588128167] 
\draw  [line width=0.75]  (401.62,67.62) .. controls (401.65,65.07) and (403.66,63.02) .. (406.1,63.05) .. controls (408.55,63.08) and (410.51,65.18) .. (410.48,67.73) .. controls (410.45,70.28) and (408.45,72.33) .. (406,72.3) .. controls (403.55,72.27) and (401.59,70.17) .. (401.62,67.62) -- cycle ;
%Shape: Ellipse [id:dp772432639402227] 
\draw   (362.37,146.57) .. controls (362.4,144.02) and (364.41,141.97) .. (366.85,142) .. controls (369.3,142.03) and (371.26,144.12) .. (371.23,146.68) .. controls (371.2,149.23) and (369.2,151.27) .. (366.75,151.24) .. controls (364.3,151.21) and (362.34,149.12) .. (362.37,146.57) -- cycle ;
%Shape: Ellipse [id:dp2696003867836012] 
\draw   (401.67,126.52) .. controls (401.7,123.97) and (403.71,121.92) .. (406.16,121.95) .. controls (408.6,121.98) and (410.56,124.08) .. (410.53,126.63) .. controls (410.51,129.18) and (408.5,131.23) .. (406.05,131.2) .. controls (403.61,131.17) and (401.65,129.07) .. (401.67,126.52) -- cycle ;
%Straight Lines [id:da7307025472213191] 
\draw [color={rgb, 255:red, 0; green, 0; blue, 0 }  ,draw opacity=1 ][fill={rgb, 255:red, 255; green, 0; blue, 0 }  ,fill opacity=1 ][line width=1.5]    (371.18,47.73) -- (402.5,64.89) ;
%Straight Lines [id:da4926319718771004] 
\draw [color={rgb, 255:red, 0; green, 0; blue, 0 }  ,draw opacity=1 ][fill={rgb, 255:red, 255; green, 0; blue, 0 }  ,fill opacity=1 ][line width=1.5]    (371.23,146.68) -- (402.5,129.89) ;
%Shape: Rectangle [id:dp6914705578640723] 
\draw   (266.8,20.5) -- (374.35,20.5) -- (374.35,176.39) -- (266.8,176.39) -- cycle ;
%Shape: Ellipse [id:dp7741627736085686] 
\draw   (362.35,106.52) .. controls (362.38,103.97) and (364.38,101.92) .. (366.83,101.95) .. controls (369.28,101.98) and (371.24,104.08) .. (371.21,106.63) .. controls (371.18,109.18) and (369.17,111.23) .. (366.73,111.2) .. controls (364.28,111.17) and (362.32,109.07) .. (362.35,106.52) -- cycle ;
%Shape: Ellipse [id:dp19408396190006139] 
\draw  [line width=0.75]  (362.35,87.62) .. controls (362.38,85.07) and (364.38,83.02) .. (366.83,83.05) .. controls (369.28,83.08) and (371.24,85.18) .. (371.21,87.73) .. controls (371.18,90.28) and (369.17,92.33) .. (366.73,92.3) .. controls (364.28,92.27) and (362.32,90.17) .. (362.35,87.62) -- cycle ;
%Straight Lines [id:da23121581072424235] 
\draw [color={rgb, 255:red, 0; green, 0; blue, 0 }  ,draw opacity=1 ][fill={rgb, 255:red, 255; green, 0; blue, 0 }  ,fill opacity=1 ][line width=1.5]    (371.21,106.63) -- (402.53,123.79) ;
%Straight Lines [id:da34885815708782486] 
\draw [color={rgb, 255:red, 0; green, 0; blue, 0 }  ,draw opacity=1 ][fill={rgb, 255:red, 255; green, 0; blue, 0 }  ,fill opacity=1 ][line width=1.5]    (370.35,84.4) -- (401.62,67.62) ;
%Straight Lines [id:da8027906045275535] 
\draw [line width=1.5]    (366.73,92.3) -- (366.83,101.95) ;
%Curve Lines [id:da8732543342061463] 
\draw [line width=1.5]  [dash pattern={on 1.69pt off 2.76pt}]  (362.32,47.62) .. controls (317,69.39) and (317,119.39) .. (362.37,146.57) ;

% Text Node
\draw (32.49,162.48) node  [font=\footnotesize,rotate=-359.02]  {$H_{i}$};
% Text Node
\draw (137.33,31.07) node [anchor=north west][inner sep=0.75pt]  [font=\footnotesize]  {$c$};
% Text Node
\draw (137.3,128.4) node [anchor=north west][inner sep=0.75pt]  [font=\footnotesize]  {$c$};
% Text Node
\draw (74.99,198.48) node  [font=\footnotesize,rotate=-359.02]  {$( a)$};
% Text Node
\draw (279.49,162.48) node  [font=\footnotesize,rotate=-359.02]  {$H_{i}$};
% Text Node
\draw (324.49,198) node  [font=\footnotesize,rotate=-359.02]  {$( b)$};

\end{tikzpicture}
		\caption{Examples of cases to consider when colouring $E(H_{i+1})\setminus E(H_i)$.}
		\label{fig:cases}
	\end{figure}
	
	Recall that by the density argument preceding this proof, configuration~\hyperref[config_A]{$(A_{\ell})$} appears at most once,~\hyperref[config_A]{$(A_{\ell-1})$} at most twice, and any~\hyperref[config_B]{$(B_{j})$} at most once.
	% Note that 
	As observed above,
	if for every $1 \le i \le t-1$, the graph~$H_{i+1}$ has
	configuration~\hyperref[config_A]{$(A_{k})$} with $2 \le k \le \ell-2$, we
	can easily avoid a rainbow $C_\ell$ by assigning, for each~$1 \le i \le t-1$, a new colour~$c_{i+1}$
	to two non-adjacent edges of~$E(H_{i+1}) \setminus E(H_i)$.
	Thus, from now on we assume that there exists at least one $H_{i+1}$ ($1\leq i\leq t-1$) with configuration \hyperref[config_A]{$(A_{\ell-1})$},~\hyperref[config_A]{$(A_{\ell})$}, or \hyperref[config_B]{$(B_{j})$} for some $3 \le j \le \ell$.
	We split our proof into a few cases, depending on the occurrence of these configurations.
	
	% \noindent\textbf{Case 1}
	\begin{case}
		\label{case_1}
		There is an index~$1 \le i_1 \le t-1$ such that~$H_{i_1+1}$
		has configuration~\hyperref[config_A]{$(A_{\ell})$}.
	\end{case}

	In this case, for all~$i \neq i_1$,~$H_{i+1}$ has configuration~\hyperref[config_A]{$(A_{2})$} or~\hyperref[config_A]{$(A_{3})$}, by the density argument.
	Moreover, at most one~$H_{i+1}$ (for some $1 \le i \le t-1$) has
	configuration~\hyperref[config_A]{$(A_{3})$}.
	
	Let $C=u_1u_2\cdots u_\ell u_1$ be an~$\ell$-cycle added to~$H_{i_1}$ to form~$H_{i_1+1}$, where~$P=u_1u_2\cdots u_\ell$ is an~$\ell$-path in~$H_{i_1}$ and~$u_\ell u_1 \notin E(H_{i_1})$.
	The number of~$\ell$-cycles in~$H_{i_1+1}$ which are not in~$H_{i_1}$ is exactly 
	the number of~$\ell$-paths in~$H_{i_1}$ with endpoints~$u_1$ and~$u_\ell$.
	
	First suppose that $P$ is the only $\ell$-path between $u_1$ and $u_\ell$ in $H_{i_1}$.
	Let $C'$ be an $\ell$-cycle in $H_ {i_1}$ that contains the edge $u_2u_3$.
	W.l.o.g.\ we may assume that $H_1=C'$.
	Then, give colour $c_1$ to two non-adjacent edges of $C'$ that are not $u_2u_3$.
	For every $H_{i+1}$ with~$1 \le i \le i_1-1$ we assign a new colour~$c_{i+1}$
	to two non-adjacent edges in~$E(H_{i+1}) \setminus E(H_i)$ (different from $u_2u_3$).
	Therefore, in step $H_{i_1+1}$, we can give a new colour $c_{i_1+1}$ to $u_1u_\ell$ and $u_2u_3$.
	Note that this partial colouring of $H_{i_1+1}$ gives two edges of the same colour in each $C_\ell$.

	Suppose that~$H_{i_1}$ contains more than one $\ell$-path between $u_1$ and $u_\ell$.
	Let $P'=u_1x_2\cdots x_{\ell-1} u_\ell$ with $P'\neq P$ be one of these paths.
	Since there is no other configuration \hyperref[config_A]{$(A_{k})$} with $k\geq 4$, one can see that~$P \cup P'$ contains
	cycle of length~$2\ell-2$, $2\ell-4$, or $\ell$.
	One can check that if~$P \cup P'$ contains an~$\ell$-cycle~$C'$,
	then~$\ell$ must be even, and~$P \cap C'$ has length~$\ell/2$.

If $P \cup P'$ forms a $(2\ell-2)$-cycle $C'$, then $C'$ appears in $H_{i_1}$ with configuration~\hyperref[config_A]{$(A_{2})$}.
W.l.o.g.\ we assume that $i_1=2$.
Then, we colour alternately the edges of $C'$ with a colour $c_1$, which implies that each of $E(P)$ and $E(P')$
contains at least two non-adjacent edges with the same colour.
Note that $H_ {2}$ may contain at most one other $\ell$-path $P''$ between $u_1$ and $u_\ell$, in which case~$\ell$ must be even
(and so $\ell \ge 6$).
But such $P''$ contains at least two consecutive edges of $P$ and two consecutive edges of $P'$ and then it must contain two edges with colour $c_1$.
Therefore, every $\ell$-cycle in $H_{i_1+1}$ is non-rainbow.

Suppose now that~$P \cup P'$ contains a~$(2\ell-4)$-cycle $C'$.
Then, $C'$ appears in $H_{i_1}$ with configuration~\hyperref[config_A]{$(A_{3})$}
(with two~$\ell$-cycles having exactly a~$3$-path in common).
We may assume w.l.o.g.\ that $x_2=u_2$,~$H_2$ has
configuration~\hyperref[config_A]{$(A_{3})$} and~$(P \cup P')-u_1 \subseteq H_2$ (note that $C'$ lies in $(P \cup P')-u_1$).
We colour the edges of $C'$ alternately with two colours~$c_1$ and~$c_2$.
If~$\ell$ is even, then there may be another~$(\ell-1)$-path $P''$ between $u_2$ and $u_\ell$ in~$H_2$
(other than~$P-u_1$ and~$P'-u_1$).
One can easily check that $P''$ must contain two edges with
the same colour ($c_1$ or $c_2$), by observing the colours given to the edges of~$C'$ which are adjacent to the endpoints of
the~$3$-path~$xyz$, where~$x,z \in C'$ and~$y$ is the unique
vertex in~$((P \cup P')-u_1)-C'$.

Now consider that~$P \cup P'$ contains an~$\ell$-cycle $C'$.
W.l.o.g.~$H_1=C'$.
Thus, we just colour the edges of~$C'$ alternately with two colours~$c_1$ and~$c_2$.
Since $\ell\geq 6$, this implies that both paths $P$ and $P'$ have two non-adjacent edges with the same colour.

% \noindent\textbf{Case 2}
\begin{case}
    \label{case_2}
 There are~$1 \le i_1 < i_2 \le t-1$ such that~$H_{i_1+1}$
 and~$H_{i_2+1}$ have configuration~\hyperref[config_A]{$(A_{\ell-1})$}.\end{case}

By the density argument, this case occurs only if~$\ell=5$ and every~$H_{i+1}$
has configuration~\hyperref[config_A]{$(A_{2})$} for~$i \neq i_1,i_2$.
Let~$C=u_1u_2u_3u_4u_5u_1$ and~$C'=v_1v_2v_3v_4v_5v_1$ be cycles where $C$ is in $H_{i_1+1}$ but not in $H_{i_1}$ and $C'$ is in $H_{i_2+1}$ but not in $H_{i_2}$.
W.l.o.g.~let~$P=u_1u_2u_3u_4$ and~$P'=v_1v_2v_3v_4$ in~$H_{i_2}$ be $4$-paths, and~$u_5 \notin V(H_{i_1})$ and~$v_5 \notin V(H_{i_2})$.

Note that~$P$ is the only~$4$-path between~$u_1$ and~$u_4$ in~$H_{i_1}$
and thus~$C$ is the only~$5$-cycle added to~$H_{i_1}$ to form~$H_{i_1+1}$.
However, it is possible that besides $P'$ there exists one other~$4$-path~$P''$ in~$H_{i_2}$ between~$v_1$ and~$v_4$.
If this is the case, then~$P' \cup P''$ contains either a~$4$-cycle or a~$6$-cycle.
This information will be useful in what follows.

We divide this proof into three parts depending on the structure of $P$ in~$H_{i_1}$:
(a) the three edges of~$P$ lie in the same~$5$-cycle, (b) exactly two consecutive edges of~$P$ lie in the same~$5$-cycle, or (c) any~$5$-cycle in~$H_{i_1}$ contains at most one edge of~$P$.

\vspace{0.3cm}
\noindent\textit{(a) the three edges of~$P$ lie in the same~$5$-cycle.}

W.l.o.g.~assume that all the edges of~$P$ lie in~$H_1$
and~$i_1=1$.
Hence~$H_1$ is of the form~$H_1=u_1u_2u_3u_4x_5u_1$ for some~$x_5$.
Note that~$C''=u_1x_5u_4u_5u_1$ is a~$4$-cycle in~$H_2$.

Suppose all the edges of $P'$ lie in~$H_2$.
Then, w.l.o.g., we may assume~$i_2=2$.
If the endpoints of~$P'$ are~$u_1$ and~$u_3$, then there is another~$4$-path $P''$ between~$u_1$ and~$u_3$ in $H_1$, say w.l.o.g.\ $P'=u_1u_5u_4u_3$ and $P''=u_1x_5u_4u_3$.
We assign a colour~$c_1$ to~$u_4x_5$,~$u_1u_2$ and~$u_3v_5$,
and a colour~$c_2$ to~$u_2u_3$,~$u_4u_5$ and~$v_5u_1$.
In this way we make all $5$-cycles in $H_3$ non-rainbow.
The case in which the ends of~$P'$ are~$u_4$ and~$u_2$ is symmetric.

For all the remainder possibilities for the endpoints of $P'$, we assign a colour $c_1$ to $u_1u_2$ and $u_4u_3$.
If the endpoints of~$P'$ are two adjacent vertices in~$V(C'')$, then we colour two non-adjacent
edges of~$C'$ with a new colour~$c_2$.
If the ends of~$P'$ are~$u_1$ and~$u_4$, then the colouring we gave to $u_1u_2$ and $u_4u_3$ already makes every $5$-cycle in $H_3$ non-rainbow.
If the endpoints of~$P'$ are~$x_5$ and a vertex in~$\{u_2,u_3\}$,
then we assign a new colour~$c_2$ to~$v_5x_5$ and~$u_2u_3$.
The case in which the ends of~$P'$ are~$u_5$ and a vertex in~$\{u_2,u_3\}$
is symmetric.
Thus, we assume that there is no $4$-path with endpoints $v_1$ and $v_4$ and all edges in $H_2$.

If at most two edges of $P'$ are in~$H_2$, then for any~$4$-path with endpoints $v_1$ and $v_4$
its edges in $H_2$ must be consecutive.
Hence we may assume w.l.o.g.~that, for $P'$, the edge~$v_3v_4$ is not in $E(H_2)$.
Because there is no triangle in~$H_{i_2}$ there can be no $6$-cycle in~$H_{i_2}$
As the unique~$4$-cycle in~$H_{i_2}$ has its edges in~$H_2$,
the $4$-path $P''$ between~$v_1$ and~$v_4$ ($P'' \neq P'$),
if it exists, contains the edge~$v_3v_4$.
Note that we can colour two non-adjacent edges of any $H_i$ with configuration~\hyperref[config_A]{$(A_{2})$} avoiding colouring the edge $v_3v_4$.
Thus, we assign a colour $c_1$ to~$u_1u_2$ and~$u_3u_4$,
and a new colour $c_2$ to~$v_3v_4$ and~$v_5v_1$.

\vspace{0.3cm}
\noindent\textit{(b) exactly two consecutive edges of~$P$ lie in the same~$5$-cycle.}

W.l.o.g.~$H_1$ contains the edges~$u_1u_2$ and~$u_2u_3$
but does not contain~$u_3u_4$.
Thus~$H_1$ is of the form~$H_1=u_1u_2u_3x_4x_5u_1$ for some~$x_4$ and~$x_5$, and~$C''=u_1x_5x_4u_3u_4u_5u_1$ is a~$6$-cycle in~$H_{i_1+1}$.
Note that~$C''$ is the only~$6$-cycle in~$H_{i_2}$,
and~$H_{i_2}$ contains no~$4$-cycle.
Hence, there are at most two~$4$-paths between~$v_1$ and~$v_4$.
If there are two such paths, they correspond to two internally disjoint paths in~$C''$.
Suppose that~$E(P') \subseteq E(C'')$.
In this case, alternately colour the edges of $C''$ with colours~$c_1$ and~$c_2$ and, for~$1 \le i \le t-1$, with~$i \neq i_1,i_2$,
assign a new colour~$c_{i+2}$ to two non-adjacent edges
in~$E(H_{i+1}) \setminus (E(H_i) \cup \{u_3u_4\})$.
Now we assume that~$E(P') \not\subseteq E(C'')$. Thus~$P'$ is the only~$4$-path
between~$v_1$ and~$v_4$ in $H_{i_2}$.
If~$E(P') \subseteq E(H_1)$ then~$E(P') \cap \{u_1u_2,u_2u_3\} \neq \emptyset$
(since $E(P')\not\subseteq E(C'')$, the path~$P'$ cannot be~$u_1x_5x_4u_3$),
and we colour~$u_4u_5$ and the two non-adjacent edges in~$E(P')$ with~$c_1$.
Assign a new colour~$c_{i+1}$ to two non-adjacent edges
in~$E(H_{i+1}) \setminus E(H_i)$, for~$1 \le i \le t-1$,~$i \neq i_1,i_2$.
Now we assume that~$E(P') \not\subseteq E(H_1)$ (possibly~$P'=P$). Therefore,~$P'$
has an edge~$v_jv_{j+1}$ with $1 \le j \le 3$ which does not belong to~$E(H_1)$.
Colour~$u_2u_3$,~$x_4x_5$ and an edge
in~$\{u_5u_1,u_5u_4\} \setminus \{v_jv_{j+1}\}$ with~$c_1$,
and give a new colour~$c_{i_2+1}$ to~$v_jv_{j+1}$ and to some edge in~$\{v_5v_1, v_5v_4\}$
not incident with~$v_j$ nor with~$v_{j+1}$.
%For~$1 \le i \le t-1$,~$i \neq i_1,i_2$, assign a new colour~$c_{i+1}$ to
%two non-adjacent edges in~$E(H_{i+1}) \setminus (E(H_i) \cup \{v_jv_{j+1}\})$.

\vspace{0.3cm}
\noindent\textit{ (c) any~$5$-cycle in~$H_{i_1}$ contains at most one edge of~$P$.}

In $H_{i_2}$ there are neither~$4$-cycles nor~$6$-cycles, and therefore~$P'$ is
the only~$4$-path between~$v_1$ and~$v_4$.
We may assume w.l.o.g.~that~$H_1$ contains~$u_2u_3$.
%and that~$u_2u_3$ is uncoloured.
If~$P'=P$, then we assign a colour $c_1$ to the edges~$u_2u_3$,~$u_5u_1$ and~$v_5u_4$, and assign a new colour~$c_{i+1}$ to two non-adjacent edges in~$E(H_{i+1}) \setminus E(H_i)$,
for~$1 \le i \le t-1$,~$i \neq i_1,i_2$.
Now we assume that~$P'\neq P$.
Since~$P'$ is the only~$4$-path in~$H_{i_2}$ between~$v_1$ and~$v_4$, we know that~$P'$ and $P$ cannot have both endpoints in common.
Therefore, w.l.o.g., we may assume that $v_1 \notin\{u_1,u_2,u_4\}$.
We assign a new colour~$c_{1}$ to the edges~$u_2u_3$ and~$u_5u_1$.
If~$v_2v_3\in\{u_2u_3,u_5u_1\}$ then colour~$v_5v_1$
with~$c_{1}$,
otherwise, colour~$v_2v_3$ and~$v_5v_1$ with a new colour~$c_{2}$.
Then, we assign a new colour~$c_{i+2}$ to two non-adjacent edges in~$E(H_{i+1}) \setminus (E(H_i) \cup \{v_2v_3\})$,
for~$1 \le i \le t-1$,~$i \neq i_1,i_2$.
%For~$1 \le i \le t-1$,~$i \neq i_1,i_2$, assign a new colour~$c_{i+1}$ to
%two nonadjacent edges in~$E(H_{i+1}) \setminus (E(H_i) \cup \{v_2v_3\})$.

% \noindent\textbf{Case 3}
\begin{case}
    \label{case_3}    
 There is exactly one~$1 \le i_1 \le t-1$ such that~$H_{i_1+1}$
 has Configuration~\hyperref[config_A]{$(A_{\ell-1})$}.
\end{case}

By the density argument,~$H_{i+1}$ has Configuration~\hyperref[config_A]{$(A_{k})$} with $2 \le k \le 4$ for all~$i \neq i_1$.
Let~$C=u_1u_2\cdots u_\ell u_1$ be a cycle where $C$ is in $H_{i_1+1}$ but not in $H_{i_1}$ and
let~$P=u_1\cdots u_{\ell-1}$ be an~$(\ell-1)$-path in~$H_{i_1}$.
The number of~$\ell$-cycles in~$H_{i_1+1}$ which are not in~$H_{i_1}$ is exactly the number of~$(\ell-1)$-paths in~$H_{i_1}$ with endpoints~$u_1$ and~$u_{\ell-1}$.
%Next we consider the case in which the number of such paths is greater than one,
%and, for this case, we define the colouring of~$H$ in such a way that
%each of these paths contains two non-adjacent edges with the same colour.
The remainder of the proof of Case~\ref{case_3} is similar to the proof of Case~\ref{case_1}, but we include it here for completeness.

First, suppose that~$P$ is the only~$(\ell-1)$-path between $u_1$ and $u_{\ell-1}$ in~$H_{i_1}$
Let~$C'$ be an~$\ell$-cycle in $H_{i_1}$ that contains the edge~$u_2u_3$.
W.l.o.g.~$H_1=C'$.
Then, give colour $c_1$ to two non-adjacent edges of $C'$ that are not $u_2u_3$.
For every $H_{i+1}$ with $1 \le i \le i_1-1$ we assign a new colour~$c_{i+1}$
to two non-adjacent edges in~$E(H_{i+1}) \setminus E(H_i)$ different from $u_2u_3$.
Therefore, in step $H_{i_1+1}$, we give a new colour $c_{i_1+1}$ to $u_1u_\ell$ and $u_2u_3$.
Note that in this partial colouring of $H_{i_1+1}$ every copy of $C_\ell$ has two non-adjacent edges of the same colour.

Suppose that~$H_{i_1}$ contains more than one $(\ell-1)$-path between $u_1$ and $u_{\ell-1}$.
Let $P'=u_1x_2\cdots x_{\ell-2} u_{\ell-1}$ with $P'\neq P$ be one of these paths.
Since there is no configuration \hyperref[config_A]{$(A_{k})$} with $k\geq 5$, one can see that~$P \cup P'$ contains an even cycle of length~$2\ell-4$, $2\ell-6$, or $\ell$.

If $P' \cup P$ forms a $(2\ell-4)$-cycle $C'$ in~$H_{i_1}$
($P'$ and~$P$ are internally disjoint), then we may assume w.l.o.g.~that~$i_1=2$, that~$H_2$ has configuration~\hyperref[config_A]{$(A_{3})$}, and~$P \cup P' \subseteq H_2$.
Then, we assign alternately colours $c_1$ and $c_2$ to the edges of $C'$.
%But note that $H_ {i_1}$ may contain another $\ell$-path $P''$ between $u_1$ and $u_{\ell-1}$ in $H_2$.
Note that if~$\ell$ is even then $H_2$ may contain another~$(\ell-1)$-path $P''$ between~$u_1$ and~$u_{\ell-1}$.
But then it is not hard to see that $P''$ contains two edges of $C'$ with the same colour.
So assume that there is no $(2\ell-4)$-cycle containing $P$.

Suppose now that~$P \cup P'$ contains a~$(2\ell-6)$-cycle $C'$.
Since there is no $(2\ell-4)$-cycle containing $P$, we may assume w.l.o.g.\ that~$x_2=u_2$,~$H_2$ has
configuration~\hyperref[config_A]{$(A_{4})$}, and~$(P \cup P')-u_1 \subseteq H_2$.
We colour the edges of $C'$ alternately with two colours~$c_1$ and~$c_2$,
and colour the two non-adjacent edges of~$C' \cap H_1$ with a new colour~$c_3$.
If~$\ell$ is even, then there may be another path $P''$ between $u_2$ and~$u_{\ell-1}$ in~$H_2$.
Such path contains the edges of~$C' \cap H_1$, and therefore have two edges with the same colour.

Now consider that~$P \cup P'$ contains an~$\ell$-cycle $C'$ (of course, we have that $\ell$ is even).
We assume that there is no~$(\ell-1)$-path~$P''$ in~$H_{i_1}$
between~$u_1$ and~$u_{\ell-1}$
such that~$P'' \cup P$ or~$P'' \cup P'$ contains a cycle with length~$2\ell-6$ or $2\ell-4$.
W.l.o.g.~$H_1=C'$.
Thus, we just colour the edges of~$C'$ alternately with two colours~$c_1$ and~$c_2$,
and we assign a new colour~$c_{i+2}$ to two non-adjacent edges
in~$E(H_{i+1}) \setminus E(H_i)$ for~$1 \le i \le t-1$,~$i \neq i_1$.

\begin{case}
    \label{case_4}
 There is~$1 \le i_1 \le t-1$ such that~$H_{i_1+1}$
 has configuration~\hyperref[config_B]{$(B_{j})$} for some $3 \le j \le \ell$.
\end{case}

By the density argument,~$H_{i+1}$ has configuration~\hyperref[config_A]{$(A_{2})$}
for all~$i \neq i_1$.
Let $C=u_1u_2\cdots u_\ell u_1$ be an~$\ell$-cycle added to~$H_{i_1}$ to form~$H_{i_1+1}$, where~$P=u_1u_2\cdots u_j$ is a~$j$-path for some $3 \le j \le \ell$, and $(\{u_3,\dots, u_\ell\} \setminus \{u_j\}) \subseteq V(H_{i+1})\setminus V(H_i)$.
If there is a path~$P'$ in~$H_{i_1}$ between~$u_1$ and~$u_j$
such that~$V(P') \cup \{u_{j+1},\dots,u_\ell\}$ induces an~$\ell$-cycle in~$H_{i_1+1}$
or there is a path~$P''$ in~$H_{i_1}$ between~$u_2$ and~$u_j$
such that~$V(P'') \cup \{u_3,\dots,u_{j-1}\}$ induces an~$\ell$-cycle in~$H_{i_1+1}$,
then~$H_{i_1+1}$ can be constructed with a construction sequence in which the last two steps
has configuration~\hyperref[config_A]{$(A_{\ell-j+3})$} and~\hyperref[config_A]{$(A_{j})$},
respectively, and therefore we have a construction sequence that we already know how to colour (see Cases~\ref{case_1},~\ref{case_2}, and~\ref{case_3}).
So we may suppose that~$H_{i_1}$ contains none of these paths, and thus we assign a new colour~$c_{i_1+1}$ to~$u_2u_3$ and~$u_\ell u_1$.
\end{proof}

\section{Cycle on four vertices} \label{sec:C4}

In this section we prove that $p_{C_4}^{\text{rb}}=n^{-3/4}$.
By a classical result of Bollobás (see~\cite{JLR00}), we know that if $p\gg n^{-3/4}$, then a.a.s.\ $G(n,p)$ contains a copy of $K_{2,4}$.
It is not hard to see that in any proper colouring of the edges of $K_{2,4}$ there is a rainbow copy of $C_4$, which implies that $p_{C_4}^{\text{rb}}\leq n^{-3/4} = n^{-m(K_{2,4})}$.

Let $G=G(n,p)$ where $p\ll n^{-3/4}$.
To prove that a.a.s. $p_{C_4}^{\text{rb}}\geq n^{-3/4}$, we define a sequence $F=C_4^1,\ldots,C_4^\ell$ of
copies of $C_4$ in $G$ as a \emph{$C_4$-chain} if for any $2\leq i\leq \ell$ we have $E(C_4^i)\cap \big( \bigcup_{j=1}^{i-1} E(C_4^i) \big) \neq \emptyset$.

We want to show that a.a.s.\ there exists a proper colouring of $G$ that contains no rainbow copy of $C_4$.
For that, consider maximal $C_4$-chains with respect to the number of $C_4$'s.
First, we colour the edges of the maximal $C_4$-chains avoiding in a way that all the $C_4$'s in such chains are non-rainbow. Then, it is enough to give new colours for each of the remaining edges (those that do not belong to the $C_4$-chains).

To colour the edges in the $C_4$-chains, from Markov's inequality and the union bound, we know that a.a.s.\ $G$ does not contain any graph $H$ with $m(H)\geq 4/3$ and $|V(H)|\leq 12$.
Let $F=C_4^1, \ldots, C_4^\ell$ be an arbitrary $C_4$-chain in $G$ with $m(F)\geq 4/3$.
Let $2\leq i\leq \ell$ be the smallest index such that $F'=C_4^1, \ldots, C_4^i$ has density $m(F')\geq 4/3$.
Then, since $F''=C_4^1, \ldots, C_4^{i-1}$ has density $m(F'')<4/3$, it is not hard to explore the structure of $G$ to conclude that $|V(F'')|\leq 10$, which implies $|V(F')|\leq 12$, as $|V(F')\setminus V(F'')|\leq 2$, a contradiction.
Therefore, a.a.s. every $C_4$-chain $F$ in $G$ satisfies $m(F)< 4/3$.

Let $F=C_4^1,\ldots,C_4^\ell$ be any $C_4$-chain in $G$ (with $m(F)<4/3$).
If we have $|V(C_4^i)\setminus V(C_4^{i-1})|=2$ for every $2\leq i\leq \ell$, then it is easy to give a new colour to two non-adjacent edges of $E(C_4^i) - E(C_4^{i-1})$, avoiding a rainbow copy of $C_4$.
Note that $F$ can have at most one $C_4^i$ such that $|V(C_4^i)\setminus V(C_4^{i-1})|=1$, as otherwise $m(F)\geq 4/3$.
But in this case, since $m(F)<4/3$, we have $\ell\leq 4$, which makes easy to colour $F$ with no rainbow copies of $C_4$.

\section{Concluding remarks}
\label{sec:con}

The problem of determining the threshold $p_H^{\text{rb}}$ for the anti-Ramsey property $G(n,p)\mcarrow H$ for graphs $H$ is far from being completely solved. 
We believe that an adaptation of the framework developed in~\cite{NePeSkSt17} and the ideas described in this paper could be useful to prove that $n^{-1/m_2(H)}$ is in fact the threshold for other classes of graphs, for example, not so small bipartite graphs $H$ (note that this is not the case for $C_4$).
One of the main direction for future research is to solve the following problem.
\begin{problem}
Determine all graphs $H$ such that $p_H^{\text{rb}} = n^{-1/m_2(H)}$.
\end{problem}

We remark that the only graphs $H$ for which the threshold is known and it is not $n^{-1/m_2(H)}$ are cycles and complete graphs on four vertices.
Thus, to determine the threshold for a large family of graphs for which it is not given by the maximum $2$-density is also an interesting problem.

%We remark that apart from $C_4$ and $K_4$, no graph $H$ is known for which the threshold is not given by $n^{-1/m_2(H)}$.
%Thus, to obtain a threshold not given by the maximum $2$-density for a large family of graphs it is also an interesting problem.

\bibliographystyle{amsplain}
\bibliography{bibliography}

\end{document}